\documentclass[11pt]{amsart}

\usepackage{amsmath,enumerate,enumitem}
\usepackage{amssymb}
\usepackage{amsthm}
\usepackage{graphicx}
\usepackage{tikz}
\usepackage{xcolor}

\allowdisplaybreaks

\usepackage[margin=3.5cm]{geometry}

\def\pr{\mathbb{P}}
\def\E{\mathbb{E}}
\def\var{\mathbb{Var}}

\def\E{\mathbb{E}}
\def\R{\mathbb{R}}

\def\1{\mathbf{1}}

\def\lam {\lambda}

\def\bv{\mathbf{v}}
\def\bT{\mathbf{T}}
\def\bX{\mathbf{X}}

\def\var{\text{var}}

\newtheorem*{theorem*}{Theorem}
\newtheorem{theorem}{Theorem}
\newtheorem{lemma}[theorem]{Lemma}

\newtheorem*{defn*}{Definition}

\newtheorem*{prop*}{Proposition}

\newtheorem*{conj*}{Conjecture}

\newcommand{\new}[1]{{#1}}

\begin{document}
\title[On kissing numbers and spherical codes in high dimensions]{On kissing numbers and spherical codes \\in high dimensions}
\author{Matthew Jenssen}
\author{Felix Joos}
\author{Will Perkins}


\thanks{This research leading to these results was supported in part by EPSRC grant EP/P009913/1 (W.~Perkins) and DFG grant JO 1457/1-1 (F.~Joos).}
\address{University of Oxford}
\email{matthew.jenssen@maths.ox.ac.uk }
\address{University of Birmingham}
\email{f.joos@bham.ac.uk}
\address{University of Birmingham}
\email{math@willperkins.org}
\date{\today}

\begin{abstract}
We prove a lower bound of $\Omega (d^{3/2}\cdot(2/\sqrt{3})^d)$ on the kissing number in dimension $d$. This improves the classical lower bound of Chabauty, Shannon, and Wyner by a linear factor in the dimension.  We obtain a similar linear factor improvement to the best known lower bound on the maximal size of a spherical code of acute angle $\theta$ in high dimensions. 
\end{abstract}

\maketitle

\section{Introduction}

\subsection{Kissing numbers}

The kissing number in dimension $d$, $K(d)$, is the maximum number of non-overlapping unit spheres that can touch a single unit sphere in dimension $d$. See Figure~\ref{Fig2dkiss} for an optimal kissing configuration in dimension $2$.   The kissing number has been determined exactly in only a small number of dimensions: $1,2,3,4,8,$ and $24$~\cite{levenstein1979bounds,musin2008kissing,odlyzko1979new,schutte1952problem}.

\begin{figure}[t]
\centering
\begin{tikzpicture}[scale=0.8]

\draw[fill] (0,0) circle (0.05);

\draw[very thick] (0,0) circle (1);

\draw[very thick, fill=black!20] (0:2) circle (1);
\draw[very thick, fill=black!20] (60:2) circle (1);
\draw[very thick, fill=black!20] (120:2) circle (1);
\draw[very thick, fill=black!20] (180:2) circle (1);
\draw[very thick, fill=black!20] (240:2) circle (1);
\draw[very thick, fill=black!20] (300:2) circle (1);

\draw[fill] (0:2) circle (0.05);
\draw[fill] (60:2) circle (0.05);
\draw[fill] (120:2) circle (0.05);
\draw[fill] (180:2) circle (0.05);
\draw[fill] (240:2) circle (0.05);
\draw[fill] (300:2) circle (0.05);

\draw[thick]
(2,0)--(0,0)--(60:2);

\draw[thick] (0.85,0) arc (0:60:0.85);

\draw (0.5,0.3) node {$\frac{\pi}{3}$};

\end{tikzpicture}
\caption{An optimal kissing configuration in dimension $2$.}
\label{Fig2dkiss}
\end{figure}
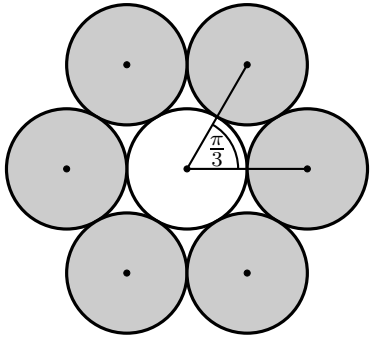

For very large dimensions,  much less is known.  There is a classical lower bound, due in various forms to~Chabauty~\cite{chabauty1953resultats}, Shannon~\cite{shannon1959probability}, and Wyner~\cite{wyner1965capabilities}, that shows
\begin{align}
\label{eqWynerKiss}
K(d) &\ge (1+o(1))  \sqrt{ \frac{3\pi d  }{8}  } \left( \frac{2}{\sqrt{3}}\right)^{d}   \, .
\end{align}

The argument is very elegant.  The centers of the kissing spheres, when projected radially onto the unit sphere, must be at angular distance at least $\pi/3$ from each other.  Now suppose we have an optimal kissing configuration in dimension $d$.  Then it must be the case that the spherical caps of angular radius $\pi/3$ cover the unit sphere, otherwise we could add another kissing sphere.  The bound in~\eqref{eqWynerKiss} is simply the reciprocal of the fraction of the surface of the unit sphere covered by a cap of angular radius $\pi/3$.  We will call this the covering lower bound. 

There is an equally simple upper bound of $2^d$ on an exponential scale, by a volume argument: the spherical caps of angular radius $\pi/6$ around the projections of the centers of the kissing spheres must be disjoint. This was improved by Rankin~\cite{rankin1955closest} to 
\begin{align*}
K(d) &\le (1+o(1)) \sqrt{\frac{\pi}{8}}  d^{3/2} \cdot 2^{d/2} \,.
\end{align*}

Rankin's bound was further improved by the breakthrough work of Kabatyanskii and Levenshtein~\cite{kabatiansky1978bounds} who applied the method of Delsarte~\cite{delsarte1972bounds} to obtain
\begin{align*}
K(d) &\le 2^{.4041... \cdot d}  \, .
\end{align*}
To the best of our knowledge there have been no further improvements to either the covering lower bound or the Kabatyanskii and Levenshtein upper bound, and the gap remains exponentially large in $d$.

 Our main result is a linear factor improvement to the lower bound.
 \begin{theorem}
\label{CorKissing}
\begin{align*}
K(d) &\ge (1+o(1))  \sqrt{ \frac{3\pi  }{8}  }  \log \frac{3}{2 \sqrt{2}} \cdot  d^{3/2} \left( \frac{2}{\sqrt{3}}\right)^{d} 
\end{align*}
as $d \to \infty$. 
\end{theorem}
The constant $\sqrt{ \frac{3\pi  }{8}  }  \log \frac{3}{2 \sqrt{2}}$ is approximately $.0639$.

It is instructive to compare the state of affairs of the kissing numbers in high dimensions to that of the maximum sphere packing density in $\R^d$.  As with the kissing numbers there is a very simple covering argument that gives a lower bound of $2^{-d}$ (this is attributed to Minkowski~\cite{Min05} who proved a slightly better lower bound of $\zeta(d) \cdot 2^{1-d}$).  There is an upper bound of $2^{-.599..d}$ due to Kabatyanskii and Levenshtein~\cite{kabatiansky1978bounds} based on their bound for spherical codes. No improvements on an exponential scale to either bound are known.

There are, however, several important works providing improvements on a smaller scale. Most significantly, Rogers~\cite{rogers1947existence} improved the asymptotic order of the lower bound by showing a lower bound of $\Omega(d \cdot 2^{-d})$.
 Rogers obtained this improvement by analyzing a random lattice packing of $\R^d$ by means of the Siegel mean-value theorem.  Subsequent work of Davenport-Rogers~\cite{davenport1947hlawka}, Ball~\cite{ball1992lower}, Vance~\cite{vance2011improved}, and Venkatesh~\cite{venkatesh2012note} has improved the leading constant of Rogers' result.  Venkatesh also obtained an additional factor of $\log \log d$ in a sparse sequence of dimensions. Finally, Cohn and Zhao~\cite{cohn2014sphere} recently obtained a constant factor improvement to the Kabatyanskii and Levenshtein upper bound, by improving the geometric argument that links spherical codes to sphere packings. 
 
Theorem~\ref{CorKissing} is an analogue of Rogers' result: both give a linear factor improvement to the simple covering lower bound.  As far as we understand, none of the lattice-based methods used for the sphere packing lower bounds mentioned above can be adapted to kissing numbers.  In fact, only very recently has it been shown by Vl{\u a}du{\c t} that the lattice kissing number is exponential in $d$~\cite{2018lattice} (albeit with a smaller base of the exponent than $2/\sqrt{3}$).  Achieving the bound in Theorem~\ref{CorKissing} or~\eqref{eqWynerKiss} for the lattice kissing number remains a challenging open problem.

\subsection{Spherical codes}  
\label{secspherecodes} 
 
 As alluded to above, kissing configurations are examples of \textit{spherical codes}.  Let $S_{d-1}$ denote the unit sphere in dimension $d$. A spherical code of angle $\theta$ in dimension $d$ is a set of unit vectors $x_1, \dots , x_k \in S_{d-1}$  so that $\langle x_i, x_j \rangle \le \cos \theta$   for all $ i \ne j $; that is, the angle between each pair of distinct vectors is at least $\theta$.  The size of such a spherical code is $k$, the number of these unit vectors, or \textit{codewords}.   We denote by $A(d,\theta)$ the size of the largest spherical code of angle $\theta$ in dimension $d$.  The kissing number can therefore be written as $K(d) = A(d, \pi/3)$. For $\theta \ge \pi/2$, Rankin~\cite{rankin1955closest} determined $A(d,\theta)$ exactly.  In what follows we will therefore consider $\theta \in (0, \pi/2)$ fixed and use standard asymptotic notation, $O(\cdot), \Omega(\cdot), \Theta(\cdot), o(\cdot)$, all as $d \to \infty$.

  For a measurable set $A \subseteq S_{d-1}$, let $s(A)$ denote the normalized surface area of $A$; that is $s(A) = \frac{\hat s(A)}{\hat s(S_{d-1})}$ where $\hat s(\cdot)$ is the usual surface area.  For $x \in S_{d-1}$, let $C_\theta(x)$ be the spherical cap of angular radius $\theta $ around $x$; that is, $C_\theta(x) = \{ y \in S_{d-1}: \langle x, y \rangle \ge \cos \theta \}$. Let $s_d(\theta) = s(C_{\theta}(x))$ be the normalized surface area of a spherical cap of angular radius $\theta$; in other words,
\begin{align*}
s_d(\theta)&= \frac{1}{\sqrt{\pi}}  \frac{ \Gamma(d/2)}{ \Gamma((d-1)/2)}\int_{0}^\theta \sin^{d-2} x  \, dx \,.
\end{align*}

The covering argument of Chabauty~\cite{chabauty1953resultats}, Shannon~\cite{shannon1959probability}, and Wyner~\cite{wyner1965capabilities} gives a general lower bound for the size of spherical codes:
\begin{align}
\label{LBShanCode}
A(d,\theta) &\ge  \frac{1}{s_d(\theta)}   = (1+o(1))  \sqrt{2 \pi d} \cdot \frac{\cos \theta}{\sin^{d-1} \theta} \,.
\end{align}
The kissing number bound~\eqref{eqWynerKiss} is an application of \new{the previous inequality} with $\theta=\pi/3$. 

  The best upper bound on $A(d,\theta)$ is due to Kabatyanskii and Levenshtein~\cite{kabatiansky1978bounds}:
 \begin{align*}
A(d,\theta) &\le e^{ \phi(\theta) d (1+o(1))},
\end{align*}
for a certain $\phi(\theta) > -\log \sin \theta$. Again to the best of our knowledge there have been no further improvements to these bounds, and so for every $\theta \in (0, \pi/2)$ the gap between the upper and lower bounds is exponential in $d$.

As with the kissing number, we improve the lower bound on the maximal size of spherical codes by a linear factor in the dimension.  To state our result, we define $q(\theta)$ to be the angular radius of the smallest spherical cap that contains the intersection of two spherical caps of angular radius $\theta$ whose centers are at angle $\theta$ (see Figure~\ref{Qfig}); that is,
\begin{align}
\label{eqQdef}
q(\theta) &=  \arcsin \left(\frac{ \sqrt{(\cos \theta -1)^2 (1+2 \cos \theta)}    }{ \sin \theta   }   \right) \,.
\end{align}
For instance, $q(\pi/3) = \arcsin (\sqrt{2/3})$.  
Crucially for our proof, $q(\theta)<\theta$ for all $\theta \in (0,\pi/2)$.  
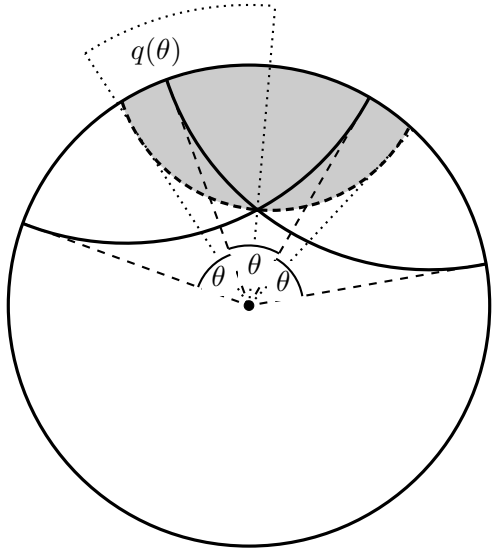
\begin{figure}[h]
\centering
\begin{tikzpicture}[scale=0.8]

\draw[fill,black!20] 
(85:4.2) +(194:2.6) arc (194:336:2.6)--(85:4.2);

\fill[even odd rule,white]
(-5.5,0)--(5.5,0) arc (0:180:5.5)
(-4,0)--(4,0) arc (0:180:4)
;

\draw[fill] (0,0) circle (0.08);
\draw[very thick] (0,0) circle (4);

\draw[thick, dashed]
(10:4)--(0,0)--(110:4)
(160:4)--(0,0)--(60:4);

\draw[thick, dotted]
(47:4)--(0,0)--(123:5)
(0,0)--(85:5)
(123:5) arc (123:85:5)
;

\node[fill=white]  at (35:0.7) {$\theta$};
\node[fill=white]  at (82:0.7) {$\theta$};
\node[fill=white]  at (136:0.7) {$\theta$};
\node  at (110:4.5) {$q(\theta)$};

\draw[very thick]
(110:6) +(248:4.6) arc (248:332:4.6)
(60:6) +(198:4.6) arc (198:282:4.6)
;
\draw[very thick, densely dashed]
(85:4.2) +(198:2.6) arc (198:332:2.6)
;

\draw[thick]
(160:0.9) arc (160:110:0.9)
(110:1) arc (110:60:1)
(60:0.9) arc (60:10:0.9)
;

\end{tikzpicture}
\caption{The angle $q(\theta)$.}
\label{Qfig}
\end{figure}

\begin{theorem}
\label{CorSpherical}
Let $\theta \in (0, \pi/2)$ be fixed. Then
\begin{align*}
A(d,\theta) &\ge  (1+o(1)) \frac{c_\theta \cdot d}{s_d(\theta)} \,
\end{align*}
as $d \to \infty$, where
\begin{align}
\label{eqCthetadef}
c_\theta &=\log \left(\frac{\sin \theta}{\sin q(\theta)  }  \right) =\log \frac{\sin ^2\theta }{\sqrt{(1- \cos \theta )^2 (1+2 \cos \theta )}}  \,.
\end{align}
\end{theorem}

Theorem~\ref{CorKissing} is obtained from Theorem~\ref{CorSpherical} by setting $\theta = \pi/3$.

\subsection{On the measure of spherical codes}

Now that we know such codes exist, we can ask how plentiful spherical codes of angle $\theta$ in dimension $d$ of size $\Theta \left(d/ s_d(\theta)   \right)$ are. We measure this by asking for the probability that a given number of random points on the unit sphere form such a code. In particular, let $E_{d,\theta,n}$ be the event that $n$ uniformly and independently chosen points from $S_{d-1}$ form a spherical code of angle $\theta$. 
\begin{theorem}
\label{thmCodeProbLB}
There exists $n = n(d) \ge (1+o(1)) \frac{ c_\theta \cdot d } {s_d(\theta)  }$ so that 
\begin{align*}
\frac{1}{n} \log ( \pr [E_{d,\theta, n}]/n!) &\ge  \log   s_d(q(\theta)) +o(1) \, .
\end{align*}
\end{theorem}

The quantity $\pr [E_{d,\theta, n}]/n!$ is simply the volume in configuration space of unordered $n$-tuples of points on the unit sphere consisting of spherical codes of angle $\theta$, and it arises naturally in the statistical physics model we use in the proof below.  As we indicate in Section~\ref{secCodesVolLB}, this lower bound is significantly larger than the trivial lower bound that comes from applying Theorem~\ref{CorSpherical} with a slightly larger angle $\theta'$ and allowing the codewords to move in a small region around their starting points (a bound of $- \Theta(d)$ compared to a bound of $-\Theta(d \log d)$).  We would be very interested to know of an alternative method to achieve a comparable bound, as it would likely have significant applications in statistical physics.

\subsection{Method of proof}
The results above follow from Theorem~\ref{thmCapAvg} below which gives a lower bound on the expected size of a random spherical code drawn from a Gibbs point process.  We choose points on the unit sphere according to a Poisson process conditioned on the event that the points form a spherical code of the desired angle; or in other words, the points are the centers of a set of non-overlapping spherical caps of angular radius $\theta/2$.  This `hard cap' model is inspired by the hard sphere and hard core lattice gas models from statistical physics, both simple models of gasses in which the only interaction between particles is a hard core exclusion.   The model comes with a parameter, $\lam$, the fugacity, which governs the expected number of caps or spheres in a configuration drawn from the model.

The proof is self-contained and uses basic tools from probability theory and one geometric lemma.  The idea of the proof comes from recent work on sphere packings and independent sets in sparse graphs~\cite{davies2018average,Fjoos17}, and in fact after accounting for differences in the relevant statistical physics models, the structure of the proof is remarkably similar in the three settings.  At the heart of the proof are two bounds on the expected size of the random spherical code, both functions of a random subset $\mathbf T$ of the unit sphere which is itself a local function of the random code.   These bounds are in tension with each other: the larger $\mathbf T$ is, the larger the first bound but the smaller the second.  By exploiting this tension we obtain the result.

\subsection{Related work}
 For surveys on spherical codes and kissing numbers see~\cite{boroczky2004finite,BDM12,cohn2016packing,conway2013sphere,pfender2004kissing,zong1998kissing}.
 
 Coxeter, Few, and Rogers~\cite{coxeter1959covering} proved that any covering of $\R^d$ with equal-sized spheres must have density $\Omega(d)$.  Using a covering argument this immediately implies a lower bound of $\Omega(d \cdot 2^{-d})$ on the sphere packing density, albeit with a worse constant than Rogers' original result.  One could hope to gain a factor $d$ over the simple lower bound for spherical codes in the same way; but proving an analogous covering lower bound for $S_{d-1}$ remains an intriguing open problem~\cite[p.~199]{boroczky2004finite}. 

In coding theory, the covering lower bound on the size of an error-correcting code is known as the Gilbert-Varshamov bound~\cite{gilbert1952comparison,varshamov1957estimate}. For binary and $q$-ary codes under Hamming distance, Jiang and Vardy~\cite{jiang2004asymptotic} improved the asymptotic order of the Gilbert-Varshamov bound using tools from extremal graph theory.  They further proposed improving the asymptotic order of the covering bound for spherical codes as an open problem; Theorem~\ref{CorSpherical} above solves this problem.    The method of Jiang and Vardy was later used by Krivelevich, Litsyn, and Vardy~\cite{krivelevich2004lower} for the sphere packing problem and could potentially be adapted with an appropriate discretization to spherical codes.

As a direction for future work, we suggest applying our statistical physics based method to other coding problems. One  could hope to obtain not only lower bounds on the maximum size of a code with a given minimal distance but also a lower bound on the number of codes near this bound.

\section{The hard cap model}

As mentioned above, a spherical code $x_1, \dots, x_k$ of angle $\theta$ corresponds to a set of non-overlapping spherical caps $C_{\theta/2}(x_1), \dots , C_{\theta/2}(x_k)$, and determining $A(d,\theta)$ is equivalent to determining the greatest number of non-overlapping spherical caps of angular radius $\theta/2$ that can be packed on $S_{d-1}$.

Let $\mathcal P_k (d,\theta) = \{ \{x_1, \dots x_k \} \in S_{d-1}^k : \langle x_i , x_j  \rangle \le  \cos \theta \, \, \forall \, i \ne j \}$ denote the set of all spherical codes of size $k$ and angle $\theta$ in dimension $d$. For $k, d, \theta$ so that $\mathcal P_k(d,\theta) \ne \emptyset$, the \textit{canonical hard cap model} with $k$ caps is simply a uniformly random spherical code from $\mathcal P_k(d,\theta)$.  The partition function of the canonical hard cap model is
\begin{align*}
\hat Z_d^\theta( k) &= \frac{1}{k!}  \int_{S_{d-1}^k}  \mathbf 1_{\mathcal D_\theta(x_1, \dots, x_k)} \, ds(x_1) \cdots ds(x_k)
\end{align*}
where $\mathcal D_\theta(x_1, \dots, x_k)$ is the event that $\langle x_i , x_j \rangle \le  \cos \theta$ for all $1 \le i < j\le k$, and the integrals over $S_{d-1}$ are with respect to the normalized surface area $s(\cdot)$.   
The probability that $k$ uniform and independent points on the unit sphere form a spherical code is simply $k! \cdot \hat Z_d^\theta(k)$.  We take $\hat Z_d^\theta(0) =1$. 

The \textit{grand canonical hard cap model} at fugacity $\lam$ is a Poisson point process $\mathbf X$ of intensity $\lam$ on $S_{d-1}$ (with normalized surface area as the underlying measure) conditioned on the event that $\langle x , y \rangle \le \cos \theta$ for all distinct $x, y \in \mathbf X$.  The partition function of the grand canonical hard cap model is 
\begin{align*}
Z_d^\theta(\lam) &= \sum_{k\ge 0} \lam^k \hat Z_d^\theta(k)  \, . 
\end{align*}
We can alternatively describe the grand canonical model as follows: choose a non-negative integer $k$ at random with probability proportional to $ \lam^k \hat Z_d^\theta(k) $, then choose a spherical code $\mathbf X$ from the canonical hard cap model with $ k$ caps.  We will write $\E_\lam [ \cdot], \pr_\lam[\cdot], \var_\lam(\cdot)$, to indicate the dependence of the model on the fugacity where needed. 

Let $\alpha_d^\theta(\lam)$ denote the expected size of a random spherical code of angle $\theta $ in dimension $d$ drawn from the grand canonical hard cap model.  That is,
\begin{align*}
\alpha_d^\theta(\lam) &=  \sum_{k\ge 0} k\cdot \pr_\lam[|\mathbf X|=k]= \sum_{k\ge 0} k \cdot \frac{\lam^k \hat Z_d^\theta(k)}{Z_d^\theta(\lam) }=\frac{\lam  \left( Z_d^\theta(\lam) \right) ^\prime}{Z_d^\theta(\lam)} = \lam \cdot \left( \log Z_d^\theta(\lam) \right) ^\prime.
\end{align*}

We note that $\alpha_d^\theta$ is strictly increasing in $\lam$ (Lemma~\ref{lemHardCap} below), but it is not true in general that $\lim_{\lam \to \infty} \alpha_d^\theta(\lam) = A(d,\theta)$ (e.g.~consider $d=2,\theta = \pi/3$).  However we can obtain $A(d,\theta)$ as a double limit: $A(d,\theta) = \lim_{\theta' \nearrow \theta} \lim_{\lam \to \infty} \alpha_{d}^{\theta'}(\lam)$.

Our main result of this section is a lower bound on $\alpha_d^\theta(\lam)$. 

\begin{theorem}
\label{thmCapAvg}
For $\lam \ge \frac{1}{d \cdot s_d(q(\theta))}$, 
\begin{align*}
\alpha_d^\theta(\lam) &\ge  (1+o(1))  \frac{c_\theta \cdot d }{s_d(\theta) } \,,
\end{align*}
where $q(\theta)$ and $c_{\theta}$ are given in~\eqref{eqQdef} and~\eqref{eqCthetadef} respectively. 
\end{theorem}

Theorems~\ref{CorKissing} and~\ref{CorSpherical} follow immediately as $A(d,\theta) \ge \alpha_d^\theta (\lam)$.  The condition on $\lam$ in Theorem~\ref{thmCapAvg} has no consequence for Theorems~\ref{CorKissing}~and~\ref{CorSpherical}, but it is directly related to the bound we obtain in Theorem~\ref{thmCodeProbLB}. 

\section{A lower bound on $\alpha_d^\theta(\lam)$}

To prove Theorem~\ref{thmCapAvg}, we define the hard cap model on a measurable subset $A \subseteq S_{d-1}$ in the natural way: $\mathbf X_A$ is a Poisson process of intensity $\lam$ on $A$ conditioned on distinct points of $\mathbf X_A$ having minimal angle at least $\theta$.  We write $\mathbf X$ for $\mathbf X_{S_{d-1}}$.   The partition function $Z_A^\theta(\lam)$ is defined as $Z_{A}^\theta(\lam) = \sum_{k \ge 0} \lam^k \hat Z_A^\theta(k) $ where
\begin{align*}
 \hat Z_A^\theta(k) &=  \frac{1}{k!}  \int_{A^k}  \mathbf 1_{\mathcal D_\theta(x_1, \dots, x_k)} \, ds(x_1) \cdots ds(x_k)
\end{align*}
and the integration is again with respect to the normalized surface area on $S_{d-1}$.  

We let $\alpha_{A}^\theta(\lam)$ the expected size of the random spherical code on $A$, and note that
\begin{align}\label{eq:alphaA}
\alpha_{A}^\theta(\lam) &= \frac{\lam  \left( Z_A^\theta(\lam) \right) ^\prime}{Z_A^\theta(\lam)} \,.
\end{align}
Furthermore, let $\mathrm{F}_A^\theta(\lam)$ be the \textit{free area}: the expected normalized surface area of the set of points at angular distance greater than $\theta$ from $\mathbf X_A$;
	that is, the set of all points $y$ in $A$ so that $\new{\bX_A} \cup \{y \}$ is still a spherical code of angle $\theta$.
	Hence
\begin{align*}
\mathrm{F}_A^\theta(\lam) &= \E_\lam \left[s \left( \{ y \in A: \langle y, x \rangle \le \cos \theta \,\, \forall x \in \mathbf X_A  \} \right)   \right ]  \, .
\end{align*}

We define a two-part experiment as follows.  Sample a spherical code $\mathbf X_A$ of angle $\theta$ from the hard cap model on $A$ at fugacity $\lam$ and independently choose a unit vector $\mathbf v$ uniformly from~$A$. 
Define the random set 
\[ \mathbf T_A = \{ x \in C_{\theta}(\mathbf v)\cap A : \langle x , y \rangle \le \cos  \theta \,\, \forall  y\in \mathbf X_A\cap C_{\theta}(\bv)^c \} ; \] 
that is, $\mathbf T_A$ is the set of all points of $A$ in the spherical cap of angular radius $ \theta$ around $\mathbf v$ that are not blocked from being in the spherical code by a vector outside the cap of angular radius $ \theta$ around $\mathbf v$.  We write $\mathbf T$ for $\mathbf T_{S_{d-1}}$.

The following lemma collects key properties of $\alpha^\theta_A(\lam)$ for the proof of Theorem~\ref{thmCapAvg}.
\begin{lemma}
\label{lemHardCap}
Let $A \subseteq S_{d-1}$ be measurable and suppose $s(A) >0$. Then the following hold. 
\begin{enumerate}
[label={\rm (\roman*)}]
\item\label{cap1} $\alpha^\theta_A(\lam)$ is strictly increasing in $\lam$. 
\item\label{cap2} $\alpha^\theta_A(\lam) = \lam \cdot \mathrm{F}_A^\theta(\lam)$.
\item\label{cap3} $\alpha^\theta_A(\lam) = \lam \cdot s(A) \cdot   \E \left[ \frac{1}{Z^\theta_{\bT_A}(\lam)}  \right ]$. 
\item\label{cap4} $\alpha_{d}^\theta(\lam) = \frac{1}{s_d(\theta)}  \cdot \E \left [\alpha^\theta_{\mathbf T}(\lam) \right ]$.
\item\label{cap5} $\log Z^\theta_A(\lam)  \le \lam \cdot s(A)$. 
\item\label{cap6} $\alpha^\theta_A(\lam) \ge \lam  \cdot s(A) \cdot e^{- \lam \cdot \E  [s(\bT_A) ]}$.
\end{enumerate}
\end{lemma}
\begin{proof}
To see~\ref{cap1}, we use~\eqref{eq:alphaA}. Hence
\begin{align}
\label{eqVar}
	\lam \cdot \alpha_A^\theta(\lam)' &=\lam\left(\frac{  \lam \left( Z_A^\theta(\lam) \right) ^\prime}{Z_A^\theta(\lam)}\right)' \\
	\nonumber
	&= \lam \cdot \left(   \frac{  Z_A^\theta(\lam)^{\prime} +  \lam Z_A^\theta(\lam)^{\prime \prime} }{Z_A^\theta(\lam)  } -  \frac{\lam\left(Z_A^\theta(\lam)^{\prime}  \right)^2   }{  Z_A^\theta(\lam)^2   } \right) \\
	\nonumber
	&= \E_\lam[ |\mathbf X_A| \new{]} + \new{\E_\lam} \left [ |\mathbf X_A|(|\mathbf X_A| -1) \right] - \left( \new{\E_\lam}[|\mathbf X_A|]  \right)^2 \\
	\nonumber
	&=  \E_\lam \left [ |\mathbf X_A|^2 \right] - \left( \new{\E_\lam}[|\mathbf X_A|]  \right)^2  \\
	\nonumber
	&= \var_\lam \left( |\mathbf  X_A| \right) > 0 \,.
\end{align}

To see \ref{cap2}, we calculate
\begin{align*}
	 \alpha_A^\theta(\lam)
	&=\sum_{k=0}^\infty (k+1)\mathbb{P}_\lam[|\bX_A|=k+1]\\
	&= \frac{1}{Z_A^\theta(\lam)}\sum_{k=0}^\infty  \int_{A^{k+1}}  \frac{\lam^{k+1}}{k!} \mathbf 1_{\mathcal D_\theta(x_0, \dots, x_{k})} \, ds(x_0) \cdots ds(x_{k})\\
	&= \frac{\lam}{Z_A^\theta(\lam)}\int_{A}
	\left(1+\sum_{k=1}^\infty  \int_{A^{k}}  \frac{\lam^{k}}{k!} \mathbf 1_{\mathcal D_\theta(x_0, \dots, x_{k})} \, ds(x_1) \cdots ds(x_{k})\right)ds(x_0)\\
	&=\lam \cdot \mathrm{F}_A^\theta(\lam).
\end{align*}

To see \ref{cap3}, we use \ref{cap2} and we compute
\begin{align*}
\alpha^\theta_A(\lam) &= \lam \cdot \mathrm{F}_A^\theta(\lam) \\
&= \lam \cdot \int_A  \pr [\max_{y \in \mathbf X} \langle v , y \rangle \le \cos  \theta]  \, ds(v)  \\
&= \lam \cdot s(A) \cdot  \E \left[\mathbf 1_{\bT_A \cap \mathbf X_A = \emptyset}     \right ]  \\
&=  \lam \cdot  s(A) \cdot \E \left[\frac{1}{\new{Z_{\bT_A}^\theta}(\lam)}     \right ]  \, .
\end{align*}
 
The last equality uses the spatial Markov property of the hard cap model: conditioned on $X\cap C_{\theta}(\bv)^c$, the distribution of $X\cap C_{\theta}(\bv)$ is exactly that of the hard cap model on the set $ \mathbf T_A$.

Next we prove \ref{cap4}.
Let $\mathbf v$ be a random chosen point on $S_{d-1}$.  Let $\mathbf X$ be a spherical code chosen independently from the hard cap model on $S_{d-1}$. Let $\mathbf T \subseteq C_{\theta}(\mathbf v)$ be the random set described above. Then
\begin{align*}
\alpha_d^{\theta}(\lam) &= \frac{1}{s_d(\theta)} \E \left [ \mathbf X \cap C_{\theta} (\mathbf v)    \right] 
= \frac{1}{s_d(\theta)} \E[ \alpha^\theta_{\mathbf T}(\lam)] \, ,
\end{align*}
where for the last equality  we again use the spatial Markov property.

For the proof of \ref{cap5}, we use the definition of $Z_A^\theta(\lam)$ and $\hat Z_A^\theta(k)$. Then 
\begin{align*}
Z_A^\theta(\lam) \le \sum_{k \ge 0 } \frac{1}{k!} s(A)^k \lam ^k = e^{\lam \cdot s(A)} \,.
\end{align*}
For \ref{cap6}, we use \ref{cap3}, \ref{cap5}, and Jensen's inequality and obtain
\begin{align*}
\alpha^\theta_A(\lam) 
\stackrel{\ref{cap3}}{=} \lam \cdot s(A) \cdot   \E \left[ \frac{1}{Z^\theta_{\bT_A}(\lam)}  \right ] 
\stackrel{\ref{cap5}}{\ge} \lam \cdot s(A) \cdot   \E \left[ e^{-\lam \cdot s(\mathbf T_A)}  \right] 
\ge \lam \cdot s(A) \cdot    e^{-\lam \cdot \E[ s(\mathbf T_A)]}   \, .
\end{align*}
\end{proof}

As the final ingredient for the proof of Theorem~\ref{thmCapAvg},
we need the following result about the intersection of an arbitrary set $A$ contained in a spherical cap
and a random cap whose center is chosen at random from~$A$.
\begin{lemma}
\label{LemCapGeom}
 Let $x \in S_{d-1}$ and $A \subseteq C_{ \theta}(x)$ be measurable with $s(A) >0$. Let $\mathbf u$ be a uniformly chosen point in $A$. Then
\begin{align}\label{eq:capEq1}
	\E[s(C_{  \theta}(\mathbf u ) \cap A) ] \le 2 \cdot s_d(q(\theta)) \,,
\end{align}
where $q(\theta)$ is defined in~\eqref{eqQdef}. In particular
\begin{align}\label{eq:capEq2}
	\alpha_A^\theta(\lam) \ge \lam\cdot s(A) \cdot e^{- \lam \cdot 2 \cdot  s_d(q(\theta))} \,.
\end{align}
\end{lemma}

\begin{proof}
We have
\begin{align*}
\E[s(C_{ \theta}(\mathbf u ) \cap A) ] &= \frac{1}{s(A)} \int_A \int_A \mathbf 1_{ \langle v , u \rangle \ge \cos  \theta  } \, ds(v) \, ds(u) \\
&= \frac{2}{s(A)} \int_A \int_A \mathbf 1_{ \langle v , u \rangle \ge \cos \theta  }  \mathbf 1_{ \langle x, v \rangle \ge \langle x , u \rangle  }\, ds(v) \, ds(u)  \\
&\le 2 \max_{u\in C_{ \theta}(x)}  \int_A \mathbf 1_{\langle v , u \rangle \ge \cos \theta  }  \mathbf 1_{ \langle x, v \rangle \ge \langle x , u  \rangle  }\, ds(v)  \\
&\le 2 \max_{u\in C_{ \theta}(x)}  s\left( C_{ \theta}(u) \cap C_{ \arccos(\langle x, u \rangle ) }(x)   \right)  \, .
\end{align*}
\new{Now we fix $u\in C_\theta(x)$ and write $\tau=\arccos(\langle x, u \rangle )$.}
Next we want to show an upper bound for $s\left( C_{ \theta}(u) \cap C_{ \new{\tau}} (x)   \right)$.
Clearly 
\begin{align}
\label{eqest1}
s\left( C_{ \theta}(u) \cap C_{ \tau }(x)   \right)\leq s\left( C_{ \tau }(x)   \right) = s_d(\tau) \, .
\end{align}
Let $\theta^*$ be defined by $\sqrt{2}\sin(\theta^*/2)=\sin(\theta/2)$. 
(Hence $\theta^*>\theta/2$.)
We use the estimate~\eqref{eqest1} as long as $\tau\in [0,\theta^*]$,
which is equivalent to the angle $yxu$ being at least $\pi/2$ in Figure~\ref{fig:pyramid}.
For all larger $\tau$,
the set $C_{ \theta}(u) \cap C_{ \tau }(x)$ is contained in a cap of angular radius $\sigma=\sigma(\theta,\tau)$ as depicted in Figure~\ref{fig:pyramid}.
It is easy to see that $\sigma(\theta,\theta^*)=\theta^*$.

We claim that whenever $\tau\in [\theta/2,\theta]$ (and in particular $\tau \in [\theta^*, \theta]$) then
\begin{align}\label{eq:sigma}
\sigma (\theta, \tau) &= \arcsin \left( \frac{  \sqrt{1+ 2 \cos^2 \tau \cos \theta - 2 \cos^2 \tau - \cos^2  \theta     }  }{ \sin \tau  }   \right ) \,.
\end{align}

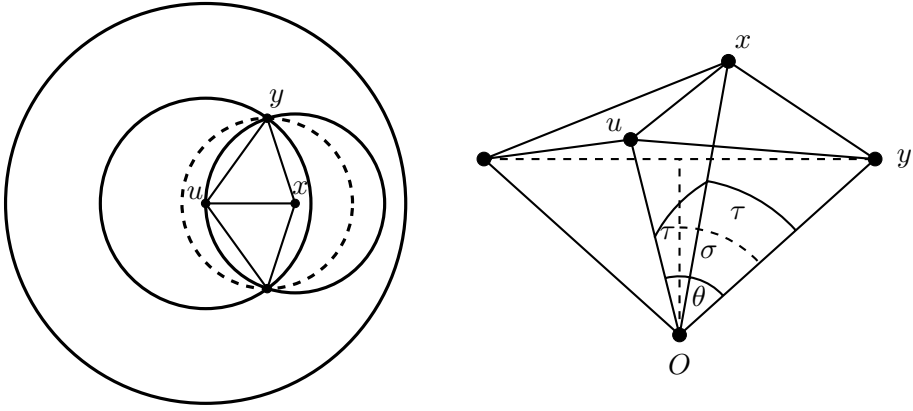
\begin{figure}[t]
\centering
\begin{tikzpicture}
\begin{scope}[scale=1.3]
\def\pot{0.07}

\node (a) at (-2,1.8) {};
\node (b) at (-0.5,2) {};
\node (c) at (0.5,2.8) {};
\node (d) at (2,1.8) {};
\node (o) at (0,0) {};

\draw[fill] (o) circle (\pot);
\draw[fill] (b) circle (\pot);
\draw[fill] (c) circle (\pot);
\draw[fill] (d) circle (\pot);
\draw[fill] (a) circle (\pot);

\draw[thick] (0,0)--(-2,1.8)--(-0.5,2)--(0,0)--(0.5,2.8)--(2,1.8)--(-0.5,2)--(0.5,2.8)--(-2,1.8)
(0,0)--(2,1.8)
(104:0.6) arc (104:42:0.6)
(79.5:1.6) arc (79.5:42:1.6)
(79.5:1.6) arc (120:153:1.4)
;

\draw[thick, dashed] 
(-2,1.8)--(2,1.8)
(0,0)--(0,1.8)
(0,1.1) arc (90:42:1.1)
;

\node  at (0.3,0.85) {$\sigma$};
\node  at (0.2,0.4) {$\theta$};
\node  at (0.6,1.25) {$\tau$};
\node  at (-0.12,1.05) {$\tau$};
\node  at (2.3,1.8) {$y$};
\node  at (0.65,3) {$x$};
\node  at (-0.67,2.15) {$u$};
\draw (0,-0.3) node {$O$};
\end{scope}

\begin{scope}[scale=0.7,shift={(-9,2.5)}]
\draw[fill] (0,0) circle (0.08);
\draw[fill] (1.7,0) circle (0.08);

\draw[very thick] (0,0) circle (3.8);
\draw[very thick] (0,0) circle (2);
\draw[very thick] (1.7,0) circle (1.7);

\draw[very thick, dashed] (1.17,0) circle (1.62);

\draw[fill] (1.17,1.62) circle (0.08);
\draw[fill] (1.17,-1.62) circle (0.08);

\draw[thick]
(0,0)--(1.17,1.62)--(1.7,0)--(1.17,-1.62)--(0,0)--(1.7,0);

\node  at (1.35,2) {$y$};
\node  at (1.8,0.3) {$x$};
\node  at (-0.2,0.2) {$u$};

\end{scope}
\end{tikzpicture}
\caption{The intersection of two caps is contained in another cap (shown dotted).
We denote by $O$ the centre of the unit sphere $S_{d-1}$.}
\label{fig:pyramid}
\end{figure}
Suppose we have a pyramid spanned by three unit vectors $u_1,u_2,u_3$ such that the angle between $u_i$ and $u_j$ is $\alpha_{ij}$.
Then the volume of the pyramid equals
$$V=\frac{1}{6}\sqrt{1+2\cos \alpha_{12}\cos \alpha_{13}\cos \alpha_{23}-\cos^2 \alpha_{12}-\cos^2 \alpha_{13}-\cos^2 \alpha_{23}}.$$
Hence the volume $V'$ of the pyramid in Figure~\ref{fig:pyramid} with apex $O$ and base $uxy$ equals
$$\frac{1}{6}\sqrt{1+2\cos^2 \tau\cos \theta-2\cos^2 \tau-\cos^2 \theta}.$$
The area of the triangle $Oux$ equals $\frac{1}{2}\sin \tau$.
Therefore,
\begin{align*}
	\frac{1}{3} \cdot \sin \sigma \cdot \frac{1}{2}\sin \tau = V'
\end{align*}
which yields~\eqref{eq:sigma}.

The function $\sigma(\theta, \tau)$ is increasing on the interval $[\theta^*,  \theta]$ and so the maximum is achieved at $\tau =  \theta$, and this maximum is larger than the maximum of the bound~\eqref{eqest1} on $[0, \theta^*]$.  Since $\sigma(\theta,\theta) =q(\theta)$, we have (for some arbitrary $x_\theta\in S_{d-1}$ with $\langle x,x_\theta\rangle=\cos \theta$)
\begin{align*}
\E[s(C_{  \theta}(\mathbf u ) \cap A) ] &\le  2  s\left( C_{ \theta}(x_{ \theta}) \cap C_{  \theta }(x)   \right)  \\
&\le 2s_d(q(\theta)) \,.
\end{align*}
Combining this with \ref{cap6} gives~\eqref{eq:capEq2}. 

\end{proof}

With the lemmas above, we prove Theorem~\ref{thmCapAvg}.  
\begin{proof}[Proof of Theorem~\ref{thmCapAvg}]
Fix $\theta \in (0, \pi/2)$ and let $\alpha = \alpha_{d}^\theta(\lam)$.   
By Jensen's inequality
\begin{align}
\label{eqPfBnd1}
\alpha &\stackrel{\ref{cap3}}{=} \lam \cdot s(S_{d-1}) \cdot \E \left[ \frac{1}{Z^\theta_{\mathbf T}(\lam)}  \right ] =\lam \cdot  \E \left[ \frac{1}{Z^\theta_{\mathbf T}(\lam)}  \right ]\ge \lam \cdot e^{-   \E [\log Z^\theta_{\mathbf T}(\lam)]} \,,
\end{align}
where the expectation is with respect to the two-part experiment forming the random set $\bT$. In words, if the random set $\mathbf T$ is small typically, where small is measured in terms of $\log Z^\theta_{\mathbf T}(\lam)$, then $\alpha$ must be large since $\mathbf v$ will be free to be part of the spherical code with good probability. 

We can also write
\begin{eqnarray*}
\alpha 
&\stackrel{\ref{cap4}}{=}& \frac{1}{s_d( \theta)}  \cdot \E \left[ \alpha^\theta_{ \bT} (\lam) \right ] \\
&\stackrel{(\ref{eq:capEq2})}{\ge}& \frac{1}{s_d( \theta)} \cdot \E \left[\lam \cdot s(\bT) \cdot e^{- \lam \cdot 2 \cdot s_d(q(\theta))}  \right ]  \\
&\stackrel{\ref{cap5}}{\ge} & 
 \frac{1}{s_d( \theta)}\cdot e^{- \lam \cdot 2 \cdot s_d(q(\theta))}  \E [ \log Z^\theta_{\bT}(\lam)] \, .
\end{eqnarray*}
In words, if $\mathbf T$ is typically large (measured again   in terms of $\log Z^\theta_{\mathbf T}(\lam)$), then $\alpha$ must be large, as the cap $C_{\theta}(\mathbf v)$ will contain many codewords in expectation.   

Now we can play these bounds against each other.  Let $z =  \E[  \log Z^\theta_{\mathbf T}(\lam)]$. Applying both bounds gives
\begin{align}
\label{eqAlpBound1}
\alpha \ge \inf_z \max \left \{ \lam e^{- z}, z \cdot  \frac{1}{s_d( \theta)} e^{ -\lam \cdot 2 \cdot s_d(q(\theta))}   \right \} .
\end{align}
As the first expression is decreasing in $z$ and the second increasing, the infimum occurs when they are equal.  In other words, $\alpha \ge \lam e^{- z^*} $, where $z^*$ is the solution to 
\begin{align*}
\lam e^{-z} &=  z \cdot   \frac{1}{s_d( \theta)} e^{ -\lam \cdot 2 \cdot s_d(q(\theta))} \,.
\end{align*}
We can write
\begin{align*}
z^* &=  W \left( \lam s_d( \theta) e^{ \lam \cdot 2 \cdot s_d(q(\theta))}  \right )
\end{align*}
where $W(\cdot)$ is the Lambert-W function, the inverse function of $f(x) = x e^x$.   One property of $W(\cdot)$ we will use is that $W(x) = \log x - \log \log x + o(1)$ as $x \to \infty$.

Although $\alpha_d^{\theta}(\lam)$ is monotonically increasing in $\lam$, the bound~\eqref{eqAlpBound1} is not. 
\new{To obtain the desired bound we must take $\lam$ asymptotically smaller than $\frac{1}{s_d(q(\theta))}$, but of the same order on a logarithmic scale. Taking $\lam = \frac{1}{d \cdot s_d(q(\theta))}$ suffices. } 

This gives, \new{recalling the asymptotics of $s_d(\theta)$ from~\eqref{LBShanCode}, }
 \begin{align*}
z^*&= W( \lam \cdot s_d(\theta) \cdot e^{2/ d})\\
&=W( \lam \cdot s_d(\theta))+o(1) \\
&=\log \lam + \log s_d(\theta) - \log\log \frac{s_d(\theta)}{ds_d(q(\theta))}+o(1)\\
 &=\log \lam + \log s_d(\theta)  - \log d - \log  c_\theta+ o(1)  \, ,
 \end{align*}
 where $c_\theta=   \log \left(\frac{\sin \theta}{\sin q(\theta)  }  \right)$,
and so
\begin{align*}
\alpha&\ge \lam e^{-z^*} = (1+o(1)) \frac{c_\theta \cdot d  }{ s_d( \theta)} 
\end{align*}
as desired.
\end{proof}

\subsection{A lower bound on the volume of spherical codes of a given size}
\label{secCodesVolLB}

\begin{proof}[Proof of Theorem~\ref{thmCodeProbLB}]
Let $\lam =  \frac{1}{s_d(q(\theta))}$  and    let $\alpha ' = (1-1/d) \alpha_d^{\theta}(\lam)$.
We know from Theorem~\ref{thmCapAvg} that 
\begin{align*}
\alpha' &= (1+o(1)) \cdot \alpha_d^{\theta}(\lam) \ge (1+o(1))\frac{ c_\theta \cdot d } {s_d(\theta)  } \,.
\end{align*}
 Then since the maximum size of a spherical code is at most $2^{d}$ for large enough $d$ (using, say, Rankin's bound), we have that 
\begin{align*}
\pr_\lam[|\mathbf X| \ge \alpha' ] &\ge \frac{\alpha_d^{\theta}(\lam)   }{d\cdot 2^{d}  }  \,.
\end{align*}
Let $n$ be the most likely value for the size of $\mathbf X$ that is at least as large as $\alpha'$.  Then we have 
\begin{align*}
\pr_\lam[|\mathbf X| = n ]  &\ge  \frac{\alpha_d^{\theta}(\lam)   }{d\cdot 2^{d} \cdot 2^{d}  } \ge 4^{- d} \,. 
\end{align*}

Now we bound the probability that $n$ unit vectors form a spherical code of angle $\theta$ from below:
\begin{align*}
\pr [E_{d,\theta, n}]/n! = \hat Z_d^\theta(n) &= \pr_\lam[ |\mathbf X| = n] \cdot Z_d^\theta(\lam) \cdot \lam^{-n}   \\
&\ge 4^{-d} \cdot Z_d^\theta(\lam) \cdot \lam^{-n} \\
&\ge 4^{-d} \cdot \lam^{-n}  \\
&= 4^{-d} \cdot   s_d(q(\theta)) ^n \, ,
\end{align*}
and so
\begin{align*}
\frac{1}{n} \log (\pr [E_{d,\theta, n}]/n!) &\ge   \log  s_d(q(\theta))   +o(1) \,.
\end{align*}
\new{Note that we could also have taken $\lam = \frac{1}{ds_d(q(\theta))}$ but this gives no improvement to the leading order term of the final bound. }
\end{proof}
To interpret this lower bound, we compare it to a naive lower bound on $\pr [E_{d,\theta, n}]$, known as the `cell model' lower bound in statistical \new{physics}.  For $\theta' > \theta$, suppose we know there exists a spherical code $x_1, \dots, \new{x_n}$ of angle $\theta'$ of size  $n \ge (1+o(1)) \frac{c_{\theta'} \cdot d}{s_d(\theta')}$.  Now consider the lower bound on $\pr [E_{d,\theta, n}]$ given as follows.  Form a `cell' of angular radius $(\theta'-\theta)/2$ around each $x_i$.  Choose points $y_1, \dots, y_n$ independently and uniformly at random and let $\tilde E$ be the event that each $y_i$ falls into a distinct cell around one of the $x_j$'s.  As the angular radius between points in distinct cells is at least $\theta$, we have $\pr[E_{d,\theta,n}] \ge \pr[\tilde E]$.  If we want $n$ to be within a constant factor of the bound from Theorem~\ref{CorSpherical}, we can only take $\theta'-\theta$ as large as $O(1/d)$, as the lower bound decreases exponentially in $d$ as the angle increases.  With $\theta'= \theta + c/d$, we have 
\begin{align*}
\pr[\tilde E] &= n! \cdot s_d\left( \frac{c}{2d} \right)^n  
\end{align*}
and so
\begin{align*}
\frac{1}{n} \log  (\pr[\tilde E]/n!) &= \log   s_d(c/2d)   \,,
\end{align*}
and since $s_d(c/2d)$ is smaller than $s_d(q(\theta))$ by a factor $d^{\Theta(d)}$, Theorem~\ref{thmCodeProbLB} gives a significant improvement to the naive lower bound.

\providecommand{\bysame}{\leavevmode\hbox to3em{\hrulefill}\thinspace}
\providecommand{\MR}{\relax\ifhmode\unskip\space\fi MR }
\providecommand{\MRhref}[2]{%
  \href{http://www.ams.org/mathscinet-getitem?mr=#1}{#2}
}
\providecommand{\href}[2]{#2}

\end{document}